\documentclass[11pt]{article}

\usepackage{amsmath,amssymb,amsthm,amscd,epsfig}
\newcommand{\vp}{\varepsilon}

\theoremstyle{plain}

\newtheorem{thm}{Theorem}

\newtheorem{lem}{Lemma}
\newtheorem{cor}{Corollary}

\theoremstyle{definition}

\newtheorem{defn}{Definition}

\theoremstyle{remark}

\begin{document}

\title{Equilibrium measures on saddle sets of holomorphic maps on $\mathbb P^2$}

\author{John Erik Fornaess and Eugen Mihailescu}
\date{}

\maketitle

\begin{abstract}

 We consider the case of hyperbolic basic sets $\Lambda$ of saddle type for holomorphic maps $f: \mathbb P^2\mathbb C \to \mathbb P^2\mathbb C$.
We study equilibrium measures $\mu_\phi$ associated to a class of H\"older potentials $\phi$ on 
$\Lambda$, and find the measures $\mu_\phi$ of iterates of arbitrary Bowen balls. Estimates for the pointwise dimension $\delta_{\mu_\phi}$ of $\mu_\phi$ that involve Lyapunov exponents and a correction term are found, and also a formula for the Hausdorff dimension of $\mu_\phi$  in the case when the preimage counting function is constant on $\Lambda$. \
 For terminal/minimal saddle sets we prove that an invariant measure $\nu$ obtained as a wedge product of two positive closed currents, is in fact the measure of maximal entropy for the \textit{restriction} $f|_\Lambda$. This allows then to obtain formulas for the measure $\nu$ of arbitrary balls, and to give a formula for the pointwise dimension and the Hausdorff dimension of $\nu$.  
\end{abstract}



\section{Introduction.}

The dynamics of holomorphic endomorphisms in higher dimensions presents many interesting geometric and ergodic aspects based on the interplay  of complex dynamics, hyperbolic smooth dynamics and ergodic theory  (see \cite{FS-survey}, \cite{Fo}, \cite{FS}, etc.)  
In this paper we study the problem of holomorphic endomorphisms of $\mathbb P^2 \mathbb C$ which are hyperbolic on basic sets of saddle type $\Lambda$ (see \cite{FS} and \cite{Ru-carte89} for hyperbolicity in the non-invertible case). 
An arbitrary holomorphic map $f: \mathbb P^2 \to \mathbb P^2$ is given by three homogeneous polynomials $[P_0: P_1: P_2]$ each of them having the same degree $d$. We will say that $d$ is the \textit{degree} (or the \textit{algebraic degree}) of $f$. 
For such maps $f: \mathbb P^2 \to \mathbb P^2$ Fornaess and Sibony have defined a positive closed current $T = \mathop{\lim}\limits_n \frac{(f^n)^*\omega}{d^n}$ which can be written locally in $\mathbb C^{3}\setminus \{0\}$ as $dd^c G$ where $G$ is the Green function associated to $f$.  This allows them to define a probability measure $\mu = T \wedge T$ which is $f$-invariant and mixing (see \cite{FS-survey}). 
In the case when $f$ is hyperbolic on a basic set $\Lambda$ one may study also the measure of maximal entropy of the \textit{restriction} of $f$ on $\Lambda$. This measure has different properties than $\mu$; for instance if $\Lambda$ is a saddle set then it has both negative and positive Lyapunov exponents. \

Given a compact $f$-invariant set $\Lambda$  one forms the \textit{natural extension} (or \textit{inverse limit}) $\hat \Lambda := \{(x, x_{-1}, x_{-2}, \ldots)\}, \  f(x_{-i}) = x_{-i+1},   x_{-i} \in \Lambda, i \ge 1\}$. The natural extension is a compact metric space with the canonical metric (see \cite{Ru-carte89}, \cite{FS}, etc.)  On the natural extension $\hat \Lambda$ there exists a shift homeomorphism $\hat f: \hat \Lambda \to \hat \Lambda$ defined by $\hat f(\hat x) = (f(x), x, x_{-1}, \ldots), \ \hat x \in \hat \Lambda$.
We denote the canonical projection by $\pi: \hat \Lambda \to \Lambda, \ \pi(\hat x) = x, \ \hat x \in \hat \Lambda$. \

\textit{Hyperbolicity for endomorphisms} is defined as a continuous splitting of the tangent bundle over $\hat \Lambda$ into stable and unstable directions (see \cite{FS}, \cite{Ru-carte89}); the stable directions depend only on base points, but unstable directions depend nevertheless on whole prehistories $\hat x \in \hat \Lambda$ (i.e past trajectories) and not only on $x$. Hyperbolic maps on basic sets 
were also studied for instance  in \cite{DJ}, \cite{M-Cam}, \cite{MU-HJM}, etc.
If $f$ is hyperbolic on $\Lambda$ then we have local stable manifolds $W^s_r(x)$ and local unstable manifolds $W^u_r(\hat x)$ where $\hat x \in \hat \Lambda$. Also notice that $\Lambda$ is not necessarily totally invariant. Thus for $x \in \Lambda$ we may have some $f$-preimages of $x$ in $\Lambda$ and others outside $\Lambda$. Moreover the number of $f$-preimages of $x$ that remain in $\Lambda$ may vary with $x$. So the endomorphism case is subtle and very different from the case of diffeomorphisms. 

Now given an arbitrary probability measure $\mu$ on a compact metric space $X$  one can define the \textit{lower pointwise dimension} and the \textit{upper pointwise dimension} at $x \in X$ respectively by:
$$
\underline{\delta}_\mu(x):= \mathop{\liminf}\limits_{\rho \to 0}
\frac{\log \mu(B(x, \rho))}{\log \rho}, \ \text{and} \
\bar{\delta}_\mu(x):= \mathop{\limsup}\limits_{\rho \to 0} \frac{\log
\mu(B(x, \rho))}{\log \rho}$$ 
In case they coincide, we call the common value $\delta_\mu(x)$ the \textit{pointwise dimension of $\mu$} at $x \in X$ (see \cite{P}).  
Also one can define the \textit{Hausdorff dimension} of $\mu$ by:
$$HD(\mu) := \inf \{HD(Z), \  Z \ \text{borelian set with} \ \mu(X \setminus Z) = 0 \}$$
In \cite{Y} Young proved that for a hyperbolic measure $\mu$ (i.e without zero Lyapunov exponents) invariated by a smooth diffeomorphism $f$ of a surface, we have $\mu$-a.e the formula $$\delta_\mu = h_\mu (\frac{1}{\chi_u(\mu)} - \frac{1}{\chi_s(\mu)}),$$ where $\chi_s(\mu), \chi_u(\mu)$ are the negative, respectively positive Lyapunov exponents of $\mu$.  \

For analytic endomorphisms $f$ on the Riemann sphere $\mathbb P^1 \mathbb C$, Manning proved in \cite{Mg} that if $f$ is hyperbolic on its Julia set $J(f)$ and has no critical points in $J(f)$, then for any ergodic $f$-invariant probability measure $\mu$ on $J(f)$ the Hausdorff dimension of $\mu$ is given by:
$$HD(\mu) = \frac{h_\mu}{\chi(\mu)},$$
where $\chi(\mu)$ is the (only) Lyapunov exponent of $\mu$. 
This formula was later extended by Mane (see \cite{Ma}) to the case of all rational maps (i.e not only hyperbolic) and invariant ergodic probabilities with positive Lyapunov exponent. \
 However the situation for higher dimensional endomorphisms and their invariant measures is different (see also \cite{FS-pbs}). In the case of polynomial endomorphisms Binder and DeMarco gave in \cite{BD} estimates for the Hausdorff dimension of the measure of maximal entropy $\mu = T\wedge T$ which involve Lyapunov exponents of $\mu$. In \cite{DD} Dinh and Dupont extended those estimates to the case of meromorphic endomorphisms of $\mathbb P^k$. 

Our case here is different in that we study the measure of maximal entropy of the \textit{restriction} of $f$ to a saddle basic set $\Lambda$, and not the measure of maximal entropy on the whole of $\mathbb P^2$. 
In fact we will consider more generally, equilibrium (Gibbs) measures for a certain class of H\"older potentials $\phi$ on $\Lambda$, such that $\phi$ satisfies an inequality relating the number of $f$-preimages remaining in $\Lambda$ and the topological pressure of $\phi$. 

In the case when $\Lambda$ is a \textit{terminal} saddle basic set, i.e when the iterates of $f$ form a normal family on $W^u(\hat \Lambda) \setminus \Lambda$, Diller and Jonsson have introduced a measure $\nu_i = \sigma^u \wedge T$ which is $f$-invariant and supported on  $\Lambda$. For the case of a \textit{minimal} saddle basic set $\Lambda$ for an s-hyperbolic map on $\mathbb P^2$, Fornaess and Sibony introduced in \cite{FS} a probability measure $\nu = T \wedge \sigma$ as  a wedge product of positive closed currents; this measure is also $f$-invariant and mixing. 
Examples of terminal sets can be obtained by perturbations of already known examples (see \cite{DJ}, \cite{FS}).

\

Our \textbf{main results} are:

First we study \textbf{equilibrium measures of H\"older potentials} on a basic set $\Lambda$. 
If the smooth endomorphism $f: \mathbb P^2\to \mathbb P^2$ is hyperbolic on $\Lambda$ and if $\phi$ is a H\"older potential on $\Lambda$, there exists a unique equilibrium measure $\mu_\phi$ for $\phi$, i.e a measure which maximizes in the Variational Principle (see \cite{Wa}, \cite{KH}, etc.)
$ 
P(\phi) = \sup \{h_\mu + \int \phi d\mu, \ \mu \ f-\text{invariant}\}$.
Equilibrium measures for stable potentials were also used for instance in \cite{M-Cam} to show that the stable dimension cannot be 2 on a basic saddle set $\Lambda$ in $\mathbb P^2$. 

In the sequel  we shall use H\"older continuous potentials $\phi$ which satisfy the following inequality
\begin{equation}\label{IN}
\phi + \log d'< P(\phi) \ \text{on} \ \Lambda,
\end{equation}
where $d'$ is an upper bound on the number of $f$-preimages in $\Lambda$ of an arbitrary point. Notice that $d'$ may even be 1, i.e the restriction to $\Lambda$ is a homeomorphism (as in the examples from \cite{MU-HJM}).  
In \textbf{Theorem \ref{ptw}} we will give precise estimates of the measure $\mu_\phi$ of an \textbf{arbitrary iterate} of a Bowen ball. 
This will help us obtain estimates and in some cases even exact formulas for the \textbf{pointwise dimension} and the \textbf{Hausdorff dimension} of $\mu_\phi$. In particular we prove that the measure $\mu_\phi$ is exact dimensional in those cases. 
In Corollary 1 we give such estimates for $\delta_{\mu_\phi}$ in the case when the number of preimages remaining in $\Lambda$ is not constant. 

Also we obtain in Corollary \ref{jac} the Jacobian (in the sense of Parry \cite{Pa}) of $\mu_\phi$ with respect to an iterate $f^m$ in the case when the preimage counting function is constant on $\Lambda$. 

We prove in \textbf{Theorem \ref{t1}} that for a terminal set $\Lambda$ the measure $\nu_i$ from above, is in fact the \textbf{measure of maximal entropy} $\mu_0$ of the restriction $f|_\Lambda$;  and if $\Lambda$ is minimal and c-hyperbolic for the Axiom A holomorphic map $f$, then the measure $\nu$ from above is equal to $\mu_0$ as well.

In \textbf{Corollaries 2, 3} we estimate, and in certain cases give formulas for the pointwise dimension of the measures $\nu_i$, $\nu$ on terminal, respectively minimal saddle sets;  in \textbf{Corollary \ref{grad2}} we give complete formulas for the pointwise dimension of $\nu$, for all the possible \textbf{minimal c-hyperbolic} saddle sets of a map of degree 2.  

In the end we will give also \textbf{examples} of holomorphic maps and equilibrium measures of H\"older potentials on terminal saddle sets, for which the (upper/lower) pointwise dimension can be estimated/computed.

\section{Equilibrium measures of H\"older potentials on saddle basic sets.}

In the sequel we consider a holomorphic map $f: \mathbb P^2 \to \mathbb P^2$, of degree $d$; this means that $f$ is given as $[f_0: f_1: f_2]$, where $f_0, f_1, f_2$ are homogeneous polynomials in coordinates $(z_0, z_1, z_2)$,  of common degree $d$.  We know also that the topological entropy of $f$ on $\mathbb P^2$ is equal to $\log d^2$ (see \cite{FS-survey}).  \

We will work on a \textit{basic set} $\Lambda$, i.e an $f$-invariant compact set $\Lambda$ for which there exists a neighbourhood $U$ such that $\mathop{\cap}\limits_{n \in \mathbb Z} f^n(U) = \Lambda$ and $f|_\Lambda$ is topologically transitive. If the non-invertible map $f$ is hyperbolic on the basic set $\Lambda$, then from the Spectral Decomposition Theorem (see \cite{Ru-carte89}, \cite{KH}) $\Lambda$ can be written as the union of finitely many mutually disjoint subsets $\Lambda_i$ s.t there exists a positive integer $m$ with $f^m(\Lambda_i) = \Lambda_i, i $ and $f^m|_{\Lambda_i}$ topologically mixing. \

Notice that we only have forward invariance of $\Lambda$, but \textbf{not} total invariance; this means that $f(\Lambda) = \Lambda$ but an arbitrary point $x \in \Lambda$ may have in general also $f$-preimages \textit{outside} $\Lambda$. 

On a basic set $\Lambda$ for $f$, let us now consider a continuous potential $\phi: \Lambda \to \mathbb R$. Then from the Variational Principle, we know that the topological pressure satisfies $P(\phi) = \sup\{h_\mu + \int \phi d\mu, \ \mu \ $f$-\text{invariant probability measure on} \ \Lambda\}$. If $\phi$ is H\"older continuous on $\Lambda$ and $f$ is hyperbolic on $\Lambda$, then there exists a \textit{unique} measure $\mu_\phi$ which attains the supremum in the Variational Principle, and it is called the \textbf{equilibrium measure} of $\phi$ (see \cite{KH} for the diffeomorphism case and \cite{M-DCDS06} for the endomorphism case). 
It follows from above that 
$
h_{\mu_\phi} + \int \phi d\mu_\phi = P(\phi).$ 

Denote by $B_n(x, \vp):= \{y \in \Lambda, d(f^i y, f^ix) < \vp, 0 \le i \le n-1\}$ a \textbf{Bowen ball}, i.e the set of points which $\vp$-follow the orbit of order $n$ of $x$. 
 
Recall now that there exists a unique correspondence between $f$-invariant measures $m$ on $\Lambda$ and $\hat f$-invariant measures $\hat m$ on $\hat \Lambda$ and that $\pi_* \hat m = m$. 
Then by going to $\hat \Lambda$ and then projecting, we showed in \cite{M-DCDS06} that the equilibrium measure $\mu_\phi$ satisfies the following estimates on Bowen balls:
\begin{equation}\label{bow}
\frac{1}{C} e^{S_n\phi(x) - nP(\phi)} \le \mu_\phi(B_n(x, \vp)) \le C e^{S_n\phi(x) - n P(\phi)},
\end{equation}
for every $n >0$, where $S_n\phi(x):= \phi(x) + \ldots + \phi(f^{n-1}(x))$ is the consecutive sum and $C$ is a positive constant independent of $x, n$. 

\begin{defn}\label{counting}
Given a basic set $\Lambda$ for the map $f$, denote by $d(x):=Card\{f^{-1}(x) \cap \Lambda\}, x \in \Lambda$ and call it the \textbf{preimage counting function} on $\Lambda$. 
\end{defn}

If $\Lambda$ is a \textit{connected} basic set such that the critical set does not intersect $\Lambda$, i.e $C_f \cap \Lambda = \emptyset$ and if there exists a neighbourhood $U$ of $\Lambda$ with $f^{-1}(\Lambda) \cap U = \Lambda$, then the preimage counting function $d(\cdot)$ is constant on $\Lambda$ (see \cite{M-Cam}). \ \ 
Notice also that the preimage counting function is \textbf{not necessarily preserved} when taking perturbations; see Example 1) at the end of paper, where by perturbing a 2-to-1 basic set we may obtain a basic set on which the restriction is 1-to-1.

Now let us recall the following Lemma proved in \cite{M-ETDS} which relates the measures of various subsets of Bowen balls, which by iteration go to the same image:

\begin{lem}\label{est}
Let $f$ be an endomorphism, hyperbolic on a basic set $\Lambda$;
consider also a Holder continuous potential $\phi$ on $\Lambda$
and $\mu_\phi$ be the unique equilibrium measure of $\phi$. Let a
small $\vp>0$, two disjoint Bowen balls $B_k(y_1, \vp), B_m(y_2,
\vp)$ and a borelian set $A \subset f^k(B_k(y_1, \vp)) \cap
f^m(B_m(y_2, \vp))$, s.t $\mu_\phi(A) >0$; denote by $A_1:=
f^{-k}A \cap B_k(y_1, \vp), A_2 := f^{-m}A \cap B_m(y_2, \vp)$ and
assume that $\mu_\phi(\partial A_1) = \mu_\phi(\partial A_2) =0$.
Then there exists a positive constant $C_\vp$ independent of $k,
m, y_1, y_2$ such that $$ \frac{1}{C_\vp} \mu_\phi(A_2) \cdot
\frac{e^{S_k \phi(y_1)}}{e^{S_m\phi(y_2)}} \cdot e^{(m-k)P(\phi)} \le
\mu_\phi(A_1) \le C_\vp \mu_\phi(A_2) \cdot
\frac{e^{S_k\phi(y_1)}}{e^{S_m\phi(y_2)}} \cdot e^{(m-k)P(\phi)}$$
\end{lem}

We shall say that two quantities $Q_1(n, x)$ and $Q_2(n, x)$ are \textbf{comparable} if there exists a positive constant $C$ such that $\frac 1C Q_1(n, x) \le Q_2(n, x) \le C Q_1(n, x)$ for all $n>0$ and $x$; this will be denoted by $Q_1 \approx Q_2$.  The constant $C$ is sometimes called a \textit{comparability constant}.
  
The next Definition is similar to that of s-hyperbolicity (see \cite{FS}), but it reffers only to a fixed basic set, not to the whole nonwandering set. 
\begin{defn}\label{c-hip}
Let $\Lambda$ be a basic set for a map $f$ on a manifold $M$. We say that $f$ is \textbf{c-hyperbolic} on $\Lambda$ if $f$ is hyperbolic on $\Lambda$,  there exists a neighbourhood $U$ of $\Lambda$ with $f^{-1}(\Lambda) \cap U = \Lambda$ and if the critical set $C_f$ of $f$ does not intersect $\Lambda$. 
\end{defn}  

Define now the set $$B(n, k, z, \vp):= f^n(B_{n+k}(z, \vp)), z \in \Lambda, \ n >0, k >0$$ This set is an iterate of a Bowen ball and when $n$ and $k$ vary, we can adjust the sides of $B(n, k, z, \vp)$ arbitrarily; in particular we can make it have (almost) equal sides in the stable and unstable directions. 

\begin{thm}\label{ptw}
Let $f: \mathbb P^2 \to \mathbb P^2$ be a holomorphic map of degree $d$ and $\Lambda$ a basic set such that $f$ is c-hyperbolic on $\Lambda$ and the preimage counting function is constant and equal to $d'$ on $\Lambda$. Consider also a Holder continuous potential $\phi$ on $\Lambda$ which satisfies $\phi(x) + \log d'< P(\phi), \forall x \in \Lambda$. Then 
$$
\mu_\phi(B(n, k, z, \vp)) \approx \frac{e^{S_{n+k}\phi(z)}}{(d')^k}
$$
Moreover the pointwise dimension of $\mu_\phi$ exists $\mu_\phi$-a.e, is denoted by $\delta_{\mu_\phi}$ and we have $\mu_\phi$-a.e:
$$
\delta_{\mu_\phi} = HD(\mu_\phi) = h_{\mu_\phi}(\frac{1}{\chi_u(\mu_\phi)} - \frac{1}{\chi_s(\mu_\phi)}) + \log d' \cdot \frac{1}{\chi_s(\mu_\phi)}
$$  
\end{thm}

\begin{proof}
Recall that $S_n\phi(y) $ is defined as the consecutive sum $\phi(y) + \phi(f(y)) + \ldots + \phi(f^{n-1}(y)), \ y \in \Lambda$.
Let us take a point $z \in \Lambda$, a positive integer $n$ and $x:= f^n(z)$. By definition $B(n, k, z, \vp) = f^n(B_{n+k}(z, \vp))$.  Now we assumed that the preimage counting function $d(\cdot)$ is constant on $\Lambda$ and equal to $d'$.  Thus every point $y$ from $\Lambda$ has exactly $d'$ $f$-preimages remaining in $\Lambda$. 

Next we assumed $\phi + \log d'< P(\phi)$; so if we define the real-valued function $\bar \phi:= \phi -P(\phi) + \log d'$, then $P(\bar \phi) = P(\phi) - P(\phi) + \log d'= \log d'$ and also $\bar \phi < 0$ on $\Lambda$. Then since $\phi$ and $\bar \phi$ are cohomologous, they have the same equilibrium measure $\mu_\phi$. Therefore we will assume in the sequel that $\phi <0$ on $\Lambda$ and $P(\phi) = \log d'$.

Consider now prehistories in $\hat \Lambda$ of an arbitrary point $y \in B(n, k, z, \vp) \subset \Lambda $; for such a prehistory $\hat y = (y, y_{-1}, y_{-2}, \ldots)$ let us denote by $n(\hat y)$ the smallest positive integer $m$ satisfying $S_m\phi(y_{m}) \le S_n\phi(z)$; since $\phi <0$ on $\Lambda$, it is clear that such an $m$ must exist for any prehistory $\hat y \in \hat \Lambda$.  So $S_{n(\hat y)-1}\phi(y_{-n(\hat y)+1}) > S_n\phi(z)$ while $S_{n(\hat y)}\phi(y_{-n(\hat y)}) \le S_n \phi(z)$. Call such a finite prehistory $(y, y_{-1}, \ldots, y_{-n(\hat y)})$ a "maximal prehistory". 

We intend to get an estimate of the measure $\mu_\phi(B(n, k, z, \vp)$ from the $f$-invariance of $\mu_\phi$ and the comparison estimates between the different pieces of its preimage set using Lemma \ref{est}. We will write $B(n, k, z, \vp)$ as a union of subsets $E$ which are contained in forward iterates of Bowen balls; the question is how these iterates intersect and what is the relation between various components of  preimage sets of different orders. \

In fact since we know that $B(n, k, z, \vp) = f^n(B_{n+k}(z, \vp))$, it means that every point in $B(n, k, z, \vp)$ has an $f^n$-preimage in $B_{n+k}(z, \vp)$. But in estimating $\mu_\phi(B(n, k, z, \vp))$ we have to consider \textit{all} $f^n$-preimages in $\Lambda$ of points from $B(n, k, z, \vp)$, in order to use the $f$-invariance of $\mu_\phi$; so  we will compare various $f^m$-preimages of subsets of $B(n, k, z, \vp)$ with the corresponding $f^n$-preimages from $B_{n+k}(z, \vp)$. Our standard for comparison of all these preimages of various orders of points in $B(n, k, z, \vp)$, will be those preimages belonging to $B_{n+k}(z, \vp)$.  

 Take an arbitrary point $y \in B(n, k, z, \vp)$ and a prehistory $\hat y \in \hat \Lambda$ of $y$; then $y \in f^n(B_{n+k}(z, \vp)) \cap f^{n(\hat y)}(B_{n(\hat y)}(y_{-n(\hat y)}, \vp))$. Let us take all the prehistories $\hat y$ of $y$ in $\hat \Lambda$; along each such prehistory we go until reaching the preimage of order $n(\hat y)$. It is clear that there exists only a finite collection $\mathcal{P}(y)$ of such maximal prehistories of $y$, since $\phi <0$ on the compact set $\Lambda$ and since we cannot continue to add indefinitely values of $\phi$ on consecutive preimages until reaching the value $S_n\phi(z)$. Denote by $E(y)$ the intersection of the iterates $f^{n(\hat y)}(B_{n(\hat y)}(y_{-n(\hat y)}, \vp))$ over all the prehistories $\hat y$ of $y$ in $\hat \Lambda$ (i.e in fact over the finite prehistories of $\mathcal{P}(y)$). 

We shall cover the set $B(n, k, z, \vp)$ with mutually disjoint subsets of various sets of type $E(y), \ y \in B(n, k, z, \vp)$. Actually we can cover $B(n, k, z, \vp)$ with a collection $\mathcal{F}$ of sets $F$, each such $F$ belonging to $E(y)$ for some $y \in B(n, k, z, \vp)$. 
Now if $F \subset E(y)$, denote by $$F(\hat y): = B(y_{-n(\hat y)}, \vp) \cap f^{-n(\hat y)}(F)$$ From the definition of $E(y)$ we know that $F = f^{n(\hat y)}(F(\hat y))$, where we recall that $n(\hat y)$ was defined above with respect to $S_n\phi(z)$. Recall also that for $y \in \Lambda$, the set of "maximal" prehistories of type $(y, y_{-1}, \ldots, y_{-n(\hat y)})$ for $ \hat y \in \hat \Lambda$ prehistory of $y$, is  denoted by $\mathcal{P}(y)$. 

Take now  two prehistories $\hat y, \hat y'$ of $y$ belonging to $\hat \Lambda$ and go along these prehistories until we reach $n(\hat y)$ and $n(\hat y')$ respectively.
We want to compare the measure $\mu_\phi$ on the preimages of $F$ along these maximal prehistories by using Lemma \ref{est}.

\begin{equation}\label{F}
\frac 1C  \mu_\phi(F(\hat y)) \frac{e^{S_{n(\hat y')}\phi(y'_{-n(\hat y')})}}{e^{S_{n(\hat y)}\phi(y_{-n(\hat y)})}} \cdot e^{(n(\hat y) -n(\hat y')) P(\phi)}  \le \mu_\phi(F(\hat y'))   \le  C \mu_\phi(F(\hat y)) \frac{e^{S_{n(\hat y')}\phi(y'_{-n(\hat y')})}}{e^{S_{n(\hat y)}\phi(y_{-n(\hat y)})}} \cdot e^{(n(\hat y) -n(\hat y')) P(\phi)}, 
\end{equation}
where $C$ is a positive constant independent of $y, \hat y, \hat y', x$. \ 
On the other hand from the definition of $n(\hat y), n(\hat y')$ we know that $|S_{n(\hat y)}\phi(y_{-n(\hat y)}) - S_n\phi(z)| \le M$ and $|S_{n(\hat y')}\phi (y'_{-n(\hat y')}) - S_n \phi(z)| \le M$, for some positive constant $M$. Therefore since $P(\phi) = \log d'$ we obtain (by taking perhaps a larger $C$) that: 
\begin{equation}\label{mu}
\frac 1C \mu_\phi(F(\hat y)) e^{(n(\hat y) - n(\hat y')) \log d'} \le \mu_\phi(F(\hat y')) \le C \mu_\phi(F(\hat y)) e^{(n(\hat y) - n(\hat y')) \log d'} 
\end{equation}
We add now the measures of various preimages $F(\hat y)$ of $F$, over all finite "maximal" prehistories from  $\mathcal{P}(y)$, in order to obtain the measure $\mu_\phi(F)$, where recall that $F \subset E(y)$.  
We take into consideration the fact that any point in $\Lambda$ has $(d')^m$ $f^m$-preimages in $\Lambda, \ m>0$. Thus if $n(\hat y)$ is say the largest maximal order associated to any prehistory from $\mathcal{P}(y)$ and if $\hat y'$ is another prehistory with $n(\hat y') = n(\hat y) - 1$, then from (\ref{mu}) it follows that $$\mu_\phi(F(\hat y)) \approx \mu_\phi(F(\hat y')) \cdot \frac {1}{d'},$$ where the comparability constant is $C$ above (i.e a positive universal constant). Now if we add the measures $\mu_\phi(F(\hat y))$ over all prehistories which coincide with $\hat y'$ up to order $n(\hat y)-1$, we obtain $\mu_\phi(F(\hat y'))$. \ 
Similarly we order the integers $n(\hat y), \hat y \in \mathcal{P}(y)$ in decreasing order and then add succesively the measures of preimages $F(\hat y)$ using (\ref{mu}) and the fact that each point has exactly $d'$ $f$-preimages in $\Lambda$. Thus if we compare the measures of $F(\hat y)$ with the measure of the $f^n$-preimage $F(f^n z, \ldots, z)$ of $F$ in $B_{n+k}(z, \vp)$,  we obtain that $$\mathop{\sum}\limits_{\hat y \in \mathcal{P}(y)} \mu_\phi(F(\hat y)) \approx \mu_\phi(F(f^nz, f^{n-1}z, \ldots, z)) \cdot (d')^n$$
But now the sets $F \in \mathcal{F}$ were chosen mutually disjoint modulo $\mu_\phi$, hence their preimages will be mutually disjoint too (recall that $C_f \cap \Lambda = \emptyset$); thus by adding over $F \in \mathcal{F}$ 
$$
\mathop{\sum}\limits_{F \in \mathcal{F}} \mathop{\sum}\limits_{\hat y \in \mathcal{P}(y), F \subset E(y)} \mu_\phi(F(\hat y))
\approx \mu_\phi(B_{n+k}(z, \vp)) \cdot (d')^n$$
Thus from the $f$-invariance of $\mu_\phi$ and by adding as above all the measures of preimages along maximal prehistories, we obtain the following formula for the measure $\mu_\phi$ of an arbitrary "rectangle" with sides in the stable and in the unstable directions centered on $\Lambda$: 
\begin{equation}\label{rectangle}
\mu_\phi(B(n, k, z, \vp)) \approx \mu_\phi(B_{n+k}(z, \vp)) \cdot (d')^n \approx \frac{e^{S_{n+k}\phi(z)}}{(d')^k},
\end{equation}
where the comparability constants are independent of $n, k, z, x$.  
Then the above formula helps us obtain the measure $\mu_\phi$ of an \textit{arbitrary ball} centered on $\Lambda$ (not only balls with radii which are indefinitely and non-uniformly  small). 

As an application we obtain the pointwise dimension of such an equilibrium measure $\mu_\phi$. 
We know from definition that $B(n, k, z, \vp) = f^n(B_{n+k}(z, \vp))$.   Since $f$ is holomorphic on $\mathbb P^2$, it follows from conformality on the local stable/unstable manifolds that the set $B(n, k, z, \vp)$ has  a side comparable to $\vp|Df_s^n(z)|$ in the stable direction, and a side comparable to $\vp|Df_u^n(z)| |Df_u^{n+k}(z)|^{-1} = \vp |Df_u^k(x)|^{-1}$ in the unstable direction. \
Now for some $n, k$ the set $B(n, k, z, \vp)$ becomes "round", i.e the stable side and the unstable side become comparable with a fixed comparability constant. So we want $\rho := \vp |Df_s^n(z)| \approx \vp|Df_u^k(x)|^{-1}$, where recall that $x = f^n(z)$. 

In general for a continuous function $\psi: \Lambda \to \mathbb R$, an $f$-invariant ergodic probability measure $\mu$ on $\Lambda$ and $\tau >0$, let us define the following set of well-behaved points with respect to $\mu$:
$$
G_n(\psi, \mu, \tau):= \{y \in \Lambda, |\frac 1n S_n\psi(y) - \int \psi d\mu| < \tau \}, \ n >0
$$
Then from Birkhoff Ergodic Theorem we have $\mu(G_n(\psi, \mu, \tau)) \to 1$ when $n \to \infty$; so for every $\tau'>0$ there is $n(\tau') >0$ such that $\mu(G_n(\psi, \mu, \mu)) > 1- \tau'$ for $n > n(\tau')$. 

We apply this to our case for the ergodic measure $\mu_\phi$ and the functions $\log |Df_s|$ and $\log |Df_u|$ which are continuous and bounded on $\Lambda$ (as $f$ has no critical points in $\Lambda$).
If $z \in G_n(\log |Df_s|, \mu_\phi, \tau)$ and $x = f^n(z) \in G_k(\log |Df_u|, \mu_\phi, \tau)$, then $$|\frac 1n S_n\log|Df_s|(z) - \int \log|Df_s| d\mu_\phi| < \tau \  \text{and} \ |\frac 1k S_k\log |Df_u|(x) - \int \log|Df_u| d\mu_\phi| < \tau$$ The question is how large is the set of such $z's$. From above it follows that $\mu_\phi(f^{-n}(G_k(\log |Df_u|, \mu_\phi, \tau)) = \mu_\phi(G_k(\log |Df_u|, \mu_\phi, \tau)) > 1- \tau'$ and $\mu_\phi(G_n(\log |Df_s|, \mu_\phi, \tau)) > 1- \tau'$, for $n > n(\tau')$. Thus for $n$ large enough:
\begin{equation}\label{G}
\mu_\phi(G_n(\log |Df_s|, \mu_\phi, \tau) \cap f^{-n}G_k(\log|Df_u|, \mu_\phi, \tau)) > 1-2\tau'
\end{equation}

We now come back to the problem of the pointwise dimension of $\mu_\phi$. It is clear from above that if $B(n, k, z, \vp)$ has comparable sides (i.e it is "round"), then $k$ must depend on $n$, so denote it by $k(z, n)$. Let us consider also  
$$
\chi_s(\mu_\phi) := \int \log |Df_s| d\mu_\phi, \ \chi_u(\mu_\phi):= \int \log |Df_u| d\mu_\phi,$$
the \textit{Lyapunov exponents} of the ergodic measure $\mu_\phi$ in the stable, respectively unstable directions; they will be denoted also by $\chi_s, \ \chi_u$ for simplicity. As $\mu_\phi$ is ergodic (see \cite{Wa}) we have that $\frac 1n S_n\log |Df_s| \mathop{\to}\limits_{n\to \infty} \chi_s$ and $\frac 1k S_k \log |Df_u| \mathop{\to}\limits_{k \to \infty} \chi_u$. Thus if $|Df_s^n(z)| \approx |Df_u^{k(z, n)}(x)|^{-1}$, it follows that  
$$
\frac{-\log C}{n} - \frac{S_{k(z, n)}\log |Df_u|(x)}{k(z, n)}\cdot \frac{k(z, n)}{n} \le \frac{S_n\log |Df_s|(z)}{n} \le \frac{\log C}{n} - \frac{S_{k(z, n)}\log |Df_u|(x)}{k(z, n)}\cdot \frac{k(z, n)}{n}
$$
Thus since $n \to \infty$, we have $\frac{k(z, n)}{n} \mathop{\to}\limits_{n \to \infty} \frac{-\chi_s}{\chi_u}$. 
But then from (\ref{rectangle}) and (\ref{G}) one sees that if $B(n, k, z, \vp)$ is a "round"ball, i.e with sides of comparable size $\rho= \vp |Df_s^n(z)|$ then for $\mu_\phi$-a.e $z \in \Lambda$ 
$$
\frac{\log \mu_\phi(B(n, k(z, n), z, \vp))}{\log  \rho} \mathop{\to}\limits_{\rho \to 0} \frac{\int\phi d\mu_\phi - \frac{\chi_s}{\chi_u}\cdot  \int \phi d\mu_\phi + \frac{\chi_s}{\chi_u} \cdot \log d'}{\chi_s}
$$
Therefore the pointwise dimension of $\mu_\phi$ is well-defined and for $\mu_\phi$-almost all $z \in \Lambda$, it is given by: 
$$
\delta_{\mu_\phi}(z) = \int \phi d\mu_\phi \cdot (\frac {1}{\chi_s} - \frac {1}{\chi_u}) + \frac{\log d'}{\chi_u}
$$
But $\mu_\phi$ is the equilibrium measure for $\phi$ and we assumed $P(\phi) = \log d'$, so $P(\phi) = \log d'= h_{\mu_\phi} + \int \phi d\mu_\phi$. Hence from above for $\mu_\phi$-a.e $z \in \Lambda$ we obtain
$$
\delta_{\mu_\phi}(z) = h_{\mu_\phi}(\frac{1}{\chi_u} - \frac{1}{\chi_s}) + \log d'\cdot \frac{1}{\chi_s},$$
 In conclusion the measure $\mu_\phi$ is exact dimensional on $\Lambda$ and it satisfies the above formula.

The fact that the Hausdorff dimension of $\mu_\phi$ takes the same value as $\delta_{\mu_\phi}$ follows from a criterion of Young (see \cite{Y}), since the pointwise dimension is constant $\mu_\phi$-a.e. 

\end{proof}

Even if the preimage counting function $d(\cdot)$ is not constant on $\Lambda$, still we obtain bounds for the measure of iterates of Bowen balls, and estimates for the lower pointwise dimension:

\begin{cor}\label{ineq}
In the setting of Theorem \ref{ptw} assume the preimage counting function satisfies $d(x) \le d'$ for $\mu_\phi$-a.e $ x \in \Lambda$ and that $\phi(x) + \log d' < P(\phi)$ for all $x \in \Lambda$; then for $\mu_\phi$-a.e $x \in \Lambda$ $$\underline{\delta}_{\mu_\phi}(x) \ge h_{\mu_\phi}(\frac{1}{\chi_u(\mu_\phi)} - \frac{1}{\chi_s(\mu_\phi)}) + \log d' \cdot \frac{1}{\chi_s(\mu_\phi)}$$
\end{cor}

\begin{proof}
In  (\ref{rectangle}) we have to consider only that now $d(x) \le d'$, therefore when adding the measures of different pieces of local backwards iterates we will obtain $\mu_\phi(B(n, k, z, \vp)) \le C \mu_\phi(B_{n+k}(z, \vp)) \cdot (d')^n \le C' \frac{e^{S_{n+k}\phi(z)}}{(d')^k},$ for some constants $C, C'$.

Hence by dividing $\log \mu_\phi(B(n, k(z, n), z, \vp))$ with the negative quantity $\log \rho$, we obtain the conclusion  as in the proof of Theorem \ref{ptw}. 

\end{proof}

In addition the above formula in (\ref{rectangle}) gives the Jacobian (in the sense of Parry \cite{Pa}) of the equilibrium measure $\mu_\phi$, with respect to an arbitrary iterate $f^ m$, by taking the iterate of a round ball $B(n, k, z, \vp)$:

\begin{cor}\label{jac}
In the same setting as in Theorem \ref{ptw}, if the preimage counting function is constant and equal to $d'$ on $\Lambda$, then  the Jacobian $J_{\mu_\phi}(f^m)$ of an arbitrary iterate satisfies $J_{\mu_\phi}(f^m) \approx (d')^m$ on $\Lambda$, where the comparability constant does not depend on $m$.
\end{cor}

In \cite{FS} Fornaess and Sibony studied \textit{s-hyperbolic} holomorphic maps on $\mathbb P^2$ and \textit{minimal} saddle basic sets, in the sense of the ordering between saddle basic sets $\Lambda_i \succ \Lambda_j$ if $W^u(\hat \Lambda_i) \cap W^s(\Lambda_j) \ne \emptyset$.  
A related notion introduced in \cite{DJ} is that of a \textit{terminal} set in the case of a holomorphic map $f$ on $\mathbb P^2$. Here $f$ is not assumed to have Axiom A and the condition refers only to $\Lambda$ itself. A saddle set $\Lambda$ is called \textit{terminal} if for any $\hat x \in \hat \Lambda$, the iterates of $f$ restricted to $W^u_{loc}(\hat x) \setminus \Lambda$ form a normal family. Notice that if $f$ is Axiom A and if $\Lambda$ is minimal,  then for any $\hat x \in \hat \Lambda$ the global unstable set $W^u(\hat x)$ does not intersect any global stable set of any other basic set, thus $W^u(\hat \Lambda) \setminus \Lambda$ is contained in the union of basins of attraction of attracting cycles. Therefore in this case minimal sets are also terminal. \ 
Examples of minimal sets for holomorphic maps on $\mathbb P^2$ are given in \cite{FS}, and examples of terminal sets are given in \cite{DJ}.

In \cite{FS}, Fornaess and Sibony constructed positive closed currents $\sigma$ on minimal sets for s-hyperbolic maps,  by using forward iterates of unstable disks (or enough disks which are transverse to local stable directions); if $D$ is an unstable disk then 
$$\frac{f^n_\star([D])}{d^n} \to \sigma \cdot \int D\wedge T$$   
Then using the positive closed  $(1, 1)$ current $\sigma$, they constructed an invariant measure $\nu$ on $\Lambda$ as $$\nu = \sigma \wedge T$$

In \cite{DJ} Diller and Jonsson introduced a positive current $\sigma^u$ by using transversal measures (see also the diffeomorphism case in \cite{RS}, \cite{S}); namely in a neighbourhood of $x \in \Lambda$,
$$<\sigma^u, \chi> = \int_{\hat W^s_{loc}}(\int_{W^u_{loc}(\hat y)}\chi) d \hat \mu^s_x(\hat y),$$
where $\hat \mu^s_x$ are transversal measures on $\hat W^s_{loc}(x):= \pi^{-1}(W^s_{loc}(x)$. Here we  use a different notation for these measures, in order to emphasize that they are supported on lifts of local stable manifolds.
If $\Lambda$ is terminal, then they defined an invariant probability measure on $\Lambda$, $$\nu_i = \sigma^u \wedge T$$  We use the notation $\nu_i$ in order to emphasize the way the current $\sigma^u$ was constructed with the help of the inverse limit. 
In general from the Spectral Decomposition Theorem a basic set can be written as a union $\Lambda = \Lambda_1 \cup \ldots \cup \Lambda_m$ of mutually disjoint compact subsets, and there exist positive integers $n_1, \ldots, n_m$ s.t $f^{n_j}$ invariates $\Lambda_j$ and $f^{n_j}$ is topologically mixing on $\Lambda_j$ (see \cite{KH}).  
In the next Theorem we want to prove that the measures $\nu, \ \nu_i$ are both equal to the measure of maximal entropy of $f|_\Lambda$ if $\Lambda$ is (topologically) mixing. 

\begin{thm} \label{t1}
 a) Let $f: \mathbb P^2 \to \mathbb P^2$ be  a holomorphic map of degree $d$ and $\Lambda$ be a terminal mixing saddle set. Then  
$\nu_i$ is equal to the measure of maximal entropy $\mu_0$ on $\Lambda$. 

b) Let $f: \mathbb P^2 \to \mathbb P^2$ be an Axiom A holomorphic map of degree $d$ and assume $f$ is c-hyperbolic on the minimal and mixing saddle basic set $\Lambda$. Then $\nu_i = \nu = \mu_0$, where $\mu_0$ is the measure of maximal entropy on $\Lambda$. 
\end{thm}

\begin{proof}

a) First let us remind the properties of the transversal measures $\hat \mu^s_x$; they are built in the same fashion as in \cite{RS} (see also \cite{S}), but on the natural extension $\hat \Lambda$. The key to that proof is the existence for diffeomorphisms of a Markov partition; in our endomorphism case, we have instead a Markov partition on the inverse limit $\hat \Lambda$ (see \cite{Ru-carte78}). 
Moreover the inverse limit $\hat \Lambda$ has local product structure, in fact it is a Smale space (see \cite{Ru-carte78}, \cite{KH}). 
One obtains then a system of transversal measures $\hat \mu^s_x$ on $\hat W^s_{loc}(x)$, where we denote  by $\hat W^s_{loc}(x)$ and $\hat W^u_{loc}(\hat x)$ the lifts to $\hat \Lambda$ of the local stable intersection $W^s_{loc}(x) \cap \Lambda$, respectively of the local unstable intersection $W^u_{loc}(\hat x) \cap \Lambda$. More precisely $\hat W^s_{loc}(x) := \pi^{-1}(W^s_{loc}(x) \cap \Lambda)$ and $\hat W^u_{loc}(\hat x):= \pi^{-1}(W^u_{loc}(\hat x) \cap \Lambda), \hat x \in \hat \Lambda$. 
Let us assume without loss of generality that all the stable and unstable local manifolds we work with, are of size $r$ for some $r>0$ small enough.  
The measures $\hat \mu^s_x$ satisfy the following properties: 

i) if $\chi^s_{x, y}: \hat W^s_r(x) \to \hat W^s_r(y)$ is the holonomy map given by $\chi^s_{x, y}(\hat \xi) = \hat W^u_r(\hat \xi) \cap \hat W^s_r(y)$, then $\hat \mu^s_x(A) = \hat \mu^s_y(\chi^s_{x, y}(A))$ for any borelian set $A$.

ii) $\hat f_\star \hat \mu^s_x = e^{h_{top}(f|_\Lambda)} \hat \mu^s_{f(x)}|_{\hat f(\hat W^s_r(x))}$

iii) \ $ \text{supp} \ \hat \mu^s_x = \hat W^s(x)$.

In fact from \cite{RS} and \cite{S} applied to our case on $\hat \Lambda$, it follows that there exist also unstable transversal measures, denoted by $\hat\mu^u_{\hat x}$ on $\hat W^u_r(\hat x), \hat x \in \hat \Lambda$ with similar properties. 
And moreover the measure of maximal entropy on $\hat \Lambda$ denoted by $\hat \mu_0$, can be written as the product of transversal stable measures $\hat \mu^s_y$ with transversal unstable measures $\hat \mu_{\hat x}$ i.e
\begin{equation}\label{product}
\hat \mu_0(\phi) = \int_{\hat W^s_r(x)} (\int_{\hat W^u_r(\hat y)} \phi \ d\hat \mu^u_{\hat y})  \ d\hat \mu^s_x(\hat y),
\end{equation}
for any function $\phi$ defined on a neighbourhood of $\hat x \in \hat \Lambda$.  

Now in \cite{DJ} the measure $\nu_i$ is defined as the wedge product $\sigma^u \wedge T$, where the positive closed current $\sigma^u$ is constructed with the help of the stable transversal measures $\hat \mu^s_x, x \in \Lambda$. Recall also that $\pi|_{\hat W^u_r(\hat x)}: \hat W^u_r(\hat x) \to W^u_r(\hat x)$ is a bijection (see \cite{M-DCDS06}), so any function $\phi$ on $\hat W^u_r(\hat x)$ determines uniquely a function denoted again with $\phi$ on $W^u_r(\hat x)$. 
Then from the measure $\nu_i$ we can form a system of measures on the lifts of local unstable manifolds $\hat W^u_r(\hat x), \hat x \in \hat \Lambda$ in the following way:
$$
\hat \nu^u_{\hat x}(\phi) = \int_{W^u_r(\hat x)} \phi T|_{W^u_r(\hat x)}
$$
We assumed that $f$ is mixing on $\Lambda$; in fact (topological) mixing of $f$ on $\Lambda$ is equivalent to mixing of $\hat f$ on $\hat \Lambda$. Define stable holonomy maps between lifts to $\hat \Lambda$ of local unstable manifolds, namely $\chi^u_{\hat x, \hat y}: \hat W^u_r(\hat x) \to \hat W^u_r(\hat y), \chi^u_{\hat x, \hat y}(\hat \xi) = \hat W^u_r(\hat y) \cap \hat W^s_r(\hat \xi), \ \hat \xi \in \hat W^u_r(\hat x)$. 

We wish to prove that the measures $\hat \nu^u_{\hat x}$ are transversal and invariant with respect to stable holonomy maps in the Smale space structure of $\hat \Lambda$, in the sense of Bowen and Marcus (\cite{BM}).
 From the way the local unstable manifolds were constructed as determined by prehistories, it follows that there is a bijection between $W^u_r(\hat x)\cap \Lambda$ and its lift $\hat W^u_r(\hat x)$ (see also \cite{M-DCDS06}). Given a borelian set $\hat A \subset \hat W^u_r(\hat x)$, there exists a unique borelian set $A \subset W^u_r(\hat x) \cap \Lambda$ such that $\pi$ is a bijection between $\hat A$ and $A$. From the definition of $\hat\nu^u_{\hat x}$, we know that $\hat \nu^u_{\hat x}(\hat A) = \int_{A \cap W^u_r(\hat x) \cap \Lambda} dd^c G|_{W^u_r(\hat x)}$.   

Let us remind some geometric properties of the positive closed current $T$ used in the definition of $\hat \nu^u_{\hat x}$ (see \cite{FS-survey}, for more details). First there exists a continuous plurisubharmonic function $G$ on $\mathbb C^3\setminus \{0\}$ called the Green function of $f$,  satisfying $G(F(z)) = d \cdot G(z)$ where $F: \mathbb C^3 \setminus \{0\} \to \mathbb C^3 \setminus \{0\}$ is the 
lift of $f$ relative to the canonical projection $\pi_2 : \mathbb C^3 \to \mathbb P^2$. We have $G \in \mathcal{P}_1$, where $\mathcal{P}_1$ is the cone of plurisubharmonic functions $u$ on $\mathbb C^3\setminus \{0\}$ satisfying the homogeneity condition $u(\lambda z) = \log |\lambda| + u(z), \  \lambda \in \mathbb C \ \text{and} \ z \in \mathbb C^3 \setminus \{0\}$.
Recall that $\pi_2^* T = dd^c G$. \

Denote now the unstable intersection $W^u_r(\hat x) \cap \Lambda$  by $Z(\hat x)$ for $\hat x \in \hat \Lambda$. Consider points $x, y$ in a subset of $\Lambda$ belonging to an open set $V \in \mathbb P^2$ so there exists a holomorphic inverse $s: V \to \mathbb C^3 \setminus \{0\}$ of $\pi_2$. Then for $r$ small we can identify $Z(\hat x), Z(\hat y)$ with their respective lifts to $\mathbb C^3 \setminus \{0\}$ for any prehistories $\hat x, \hat y \in \hat \Lambda$. Since there are no critical points of $f$ in $\Lambda$ and since we work on $\Lambda$, it follows that $Z(\hat x)$ can be split into mutually disjoint subsets on which $f^n$ is injective, i.e $Z(\hat x) = \mathop{\cup}\limits_i  Z_{i, n}(\hat x)$, \  $f^n|_{Z_{i, n}(\hat x)}: Z_{i, n}(\hat x) \to Z_i^n(\hat x)$ is bijective, and moreover $Z^n_i(\hat x), i$ are mutually disjoint. It follows that $f^n(Z(\hat x)) = \cup_i Z_i^n(\hat x)$.   Now if $Z(\hat x)$ is contained in $V$, then $f^n(Z(\hat x))$ may not be contained in $V$; but, if $f^n(Z(\hat x))$ is contained say in $V_1 \cup V_2$ where $V_1, V_2$ are open sets in $\mathbb P^2$ as above, with respective local inverses $s_1, s_2$ of $\pi_2$, and if $V_1 \cap V_2 \ne \emptyset$, then there exists a holomorphic function $\rho$ on $V_1 \cap V_2$ so that $s_1 = \rho s_2$ on $V_1 \cap V_2$.  So $dd^c (G \circ s_1) = dd^c(G(\rho s_2)) = dd^c \log |\rho| + dd^c (G \circ s_2) = dd^c(G \circ s_2)$; this implies that working with $dd^c G$ on $\mathbb C^3\setminus \{0\}$ is the same as working on $\mathbb P^2$.
Now $G \circ F = d \cdot G$ and $f^n: Z_{i, n}(\hat x) \to Z^n_i(\hat x)$ is bijective hence similarly as in \cite{FS}, $\int_{Z_{i}^n(\hat x)} dd^c G = d^n \int_{Z_{i, n}(\hat x)} dd^cG$. Thus by adding over all the indices $i$ we obtain: 
\begin{equation}\label{G}
\int_{f^n(Z(\hat x))} dd^c G = d^n \int_{Z(\hat x)} dd^c G
\end{equation}
Now let $x, y \in \Lambda$ closer than $r/2$ and iterate $Z(\hat x) $ and $Z(\hat y)$ for some prehistories $\hat x, \hat y \in \hat \Lambda$. We also take as above, the subsets $Z_{i, n}(\hat y)$ such that $f^n: Z_{i, n}(\hat y) \to Z^n_i(\hat y)$ is a bijection,  $Z^n_i(\hat y), i$
are mutually disjoint and $f^n(Z(\hat y) = \cup Z^n_i(\hat y)$. If $Z_{i, n}(\hat x), Z_{i, n}(\hat y)$ has diameter small enough, then it follows that $Z^n_i(\hat x), Z^n_i(\hat y)$ both have diameter bounded above by $r$ and they are very close to each other, in fact $d(Z^n_i(\hat x), Z^n_i(\hat y)) \to 0$ for each $i$, when $n \to \infty$. This follows as in the Laminated Distortion Lemma (see \cite{M-Cam}) since the distances between iterates of points on stable manifolds decrease exponentially, and the unstable derivative $|Df_u|$ is H\"older continuous. Now if $\psi$ is a smooth test function equal to 1 on a fixed neighbourhood of $Z^n_i$ we have $\int_{Z^n_i(\hat x)}  dd^c G = \int_{Z^n_i(\hat x)} \psi dd^c G = \int_{Z^n_i(\hat x)} G dd^c \psi$ hence since $dd^c \psi $ is continuous and $Z^n_i(\hat x)$ and $Z^n_i(\hat y)$ are close, we obtain $|\int_{f^n(A) \cap Z^n_i(\hat x)} dd^cG - \int_{f^n(\chi^u_{\hat x, \hat y}(A)) \cap Z^n_i(\hat y)} dd^cG| \le \vp  m_2(Z^n_i(\hat x))$ for $n$ large enough, where $m_2$ is the Lebesgue measure on $\mathbb P^2$.  
Now we add these inequalities over $i$ and use the fact proved in Proposition 5.3 of \cite{FS} that $m_2(f^n(Z(\hat x))) \le C d^n, n >0$. Hence by dividing with $d^n$, using (\ref{G}), and letting $n \to \infty$ we obtain 
$\int_{A \cap Z(\hat x)} dd^c G = \int_{\chi^u_{\hat x, \hat y}(A) \cap Z(\hat y)} dd^c G$. We lift then to the natural extension, keeping in mind that there exists a homeomorphism between $Z(\hat x)$ and $\hat W^u_r(\hat x)$. Hence on $\hat \Lambda$ we have:
$$
\hat \nu^u_{\hat x}(\hat A) = \hat \nu_{\hat y}(\chi^u_{\hat x, \hat y}(\hat A)), \  \hat A \ \text{borelian set in} \ \hat W^u_r(\hat x)
$$ 
The above equality can be extended next to general borelian sets contained in global unstable sets $\hat W^u(\hat x) = \mathop{\cup}\limits_{n \ge 0} \hat f^n(\hat  W^u_r(\hat x)), \hat x \in \hat \Lambda$. \
Thus by a theorem of Bowen and Marcus (\cite{BM}) extended to the mixing homeomorphism $\hat f$ on $\hat \Lambda$, it follows that there exists a positive constant $\gamma$ such that $\hat \nu^u_{\hat x} = \gamma \hat \mu^u_{0, \hat x}$,  for any $\hat x \in \hat \Lambda$,  where $\hat \mu^u_{0, \hat x}$ are the transversal measures given by the measure of maximal entropy $\hat \mu_0$ on $\hat \Lambda$ (as in  \cite{RS}, \cite{S}); see also (\ref{product}). 
In fact if $\mu_0$ is the unique measure of maximal entropy on $\Lambda$ and if $\hat \mu_0$ is the unique measure of maximal entropy on $\hat \Lambda$, then $$\mu_0 = \pi_* \hat \mu_0 \  \text{and} \  h_{\mu_0} = h_{top}(f|_\Lambda) = h_{top}(\hat f|_{\hat \Lambda}) = h_{\hat \mu_0}$$
The measure $\nu_i$ is constructed with the transversal stable measures $\hat \mu^s_x$ (which we denote also by $\hat \mu^s_{0, x}$). 
Now from \cite{Ru-carte78} we know that any $f$-invariant measure $\mu$ on $\Lambda$ can be lifted uniquely to an $\hat f$-invariant measure $\hat \mu$ on $\hat \Lambda$ such that $\pi_* \hat \mu = \mu$. In our case we denote by $\hat \nu_i$ this unique lift of $\nu_i$ to $\hat \Lambda$. 
Since both $\hat \nu_i$ and $\hat \mu_0$ are ergodic probabilities on $\hat \Lambda$, it follows that $\gamma = 1$ and that $$\hat \mu_0 = \hat \nu_i, \ \text{hence} \ \mu_0 = \nu_i$$

\ \ b) Let us assume now that $f$ has Axiom A, that $\Lambda$ is a minimal basic set (i.e the unstable set of $\Lambda$ does not intersect the stable set of any other basic set $\Lambda'$) and that $f$ is c-hyperbolic on $\Lambda$.
Then Fornaess and Sibony (\cite{FS}) constructed a positive closed $(1, 1)$ current $\sigma$ supported on the global unstable set $W^u(\hat \Lambda)$ such that if $D$ is a local disk transverse to the stable direction, then $\frac{f^n_* ([D])}{d^n} \to (\int [D]\wedge T) \sigma$. Without loss of generality assume that the disk $D$ is chosen such that $\int [D] \wedge T = 1$; and also that $T$ has no mass on the boundary $\partial D$ of $D$. 

We have from \cite{FS} that locally on a neighbourhood $\Delta$ of a point $x \in \Lambda$ there exists a measure $\lambda$ defined on the space of holomorphic maps from a local unstable disk $\Delta_1$ to a local stable disk $\Delta_2$ such that $\sigma = \int [W^u_r(\hat y)] d\lambda(g_{\hat y})$ where $W^u_r(\hat y)$ are local unstable manifolds intersecting $\Delta$ and $g_{\hat y}:\Delta_1 \to \Delta_2$ is a holomorphic map whose graph is $W^u_r(\hat y)$. 

Then $\nu = \sigma \wedge T$ is supported only on $\Lambda$; hence we can define measures $\hat \nu^s_x$ on $\hat W^s_r(x)$ by $\hat \nu^s_x(\hat A) = \lambda(\{g_{\hat y}, \ \hat y \in \hat A\}$.

Thus from the way the function $g_{\hat y}$ was defined, namely as a function whose graph is $W^u_r(\hat y)$, it follows that these measures are invariant to the local holonomy map between $\hat W^s_r(x)$ and $\hat W^s_r(y)$ for $x, y$ close. Also by covering with small flow boxes it follows we can extend this property globally. Therefore from \cite{BM} we obtain that 
$\hat \nu^s_x = \gamma \hat \mu^s_{0, x},$
 where the constant $\gamma>0$ does not depend on $x \in \Lambda$. Now $\nu$ was defined as integration of $T$ on local unstable manifolds followed by integration with respect to transversal measures; by using a) we obtain that 
$\hat \nu = \hat \mu_0, \ \text{and thus} \ \nu = \mu_0$.
Hence on minimal saddle basic sets the measure $\nu$ is equal to the measure $\nu_i$ and both are equal to the measure of maximal entropy $\mu_0$ on $\Lambda$. 

\end{proof}

Now that we know that the measure $\nu_i$ is equal to the measure of maximal entropy $\mu_0$, and to $\nu$ when $f$ has Axiom A and $\Lambda$ is minimal, we find its pointwise dimension. 

\begin{cor}\label{hol}
a) Let $\Lambda$ be a mixing terminal saddle set for a holomorphic map $f : \mathbb P^2 \to \mathbb P^2$ of degree $d$, s.t $\Lambda$ does  not intersect the critical set $C_f$ of $f$.
If each point in $\Lambda$ has at most $d'$ $f$-preimages in $\Lambda$ and if $d'< d$, then for $\mu_\phi$-a.e $z$, $$\underline{\delta}_{\nu_i}(z) \ge \log d \cdot (\frac{1}{\chi_u(\nu_i)} - \frac{1}{\chi_s(\nu_i)}) + \log d' \cdot \frac{1}{\chi_s(\nu_i)}$$

b) If $\Lambda$ is a mixing terminal saddle set for a holomorphic map $f$ on $\mathbb P^2$ of degree $d$, if $C_f \cap \Lambda = \emptyset$  and if the preimage counting function is constant equal to $d'$ on $\Lambda$ for $d'\le d$,  then we have:
$$
\delta_{\nu_i} = HD(\nu_i) = \log d \cdot (\frac{1}{\int \log |Df_u| d\nu_i} - \frac{1}{\int \log |Df_s| d\nu_i}) + \log d'\cdot \frac{1}{\int\log |Df_s| d\nu_i},$$

c)  Let $f: \mathbb P^2 \to \mathbb P^2$ be a holomorphic Axiom A map of degree $d$, which is c-hyperbolic on a connected minimal saddle set $\Lambda$. Then 
$$\delta_{\nu} = HD(\nu) = \log d \cdot (\frac{1}{\chi_u(\nu)} - \frac{1}{\chi_s(\nu)}) + \log d'\cdot \frac{1}{\chi_s(\nu)},$$
where $d'$ denotes the constant number of $f$-preimages in $\Lambda$ of a point.
\end{cor}

\begin{proof}
 \ a) We use the result of \cite{DJ} that the topological entropy of $f|_\Lambda$ is equal to $\log d$ if $\Lambda$ is a terminal saddle basic set.  
From Theorem \ref{t1} we have that  $\nu_i = \mu_0$, the measure of maximal entropy on $\Lambda$; hence $h_{\mu_0} = h_{top}(f|_\Lambda)$. Then the inequality follows from Corollary \ref{ineq} in case the number of preimages in $\Lambda$ is bounded above by $d'$.

\ b) If $\Lambda$ is terminal, then from Theorem \ref{t1} we know that $\nu_i$ is equal to $\mu_0$ the measure of maximal entropy of $f|_\Lambda$; also since the topological entropy of $f|_\Lambda$ is $\log d$, it follows that $h_{\nu_i} = \log d$. 

If the preimage counting function is constant and equal to $d'$, then in case $d' < d$, we have that condition $\phi + \log d'< P(\phi)$ is satisfied on $\Lambda$ for $\phi \equiv 0$. So one can apply Theorem \ref{ptw} in order to obtain the pointwise dimension of $\nu_i$; and since $\delta_{\nu_i}$ is constant, then $HD(\nu_i) = \delta_{\nu_i}$.  

There remains only the case $d'= d$. In this case every point in $\Lambda$ has $d$ $f$-preimages in $\Lambda$ and $h_{top}(f|_\Lambda) = \log d$ (from \cite{DJ}). \ 
Thus the unique zero of the function $t \to P(t\log |Df_s| - \log d)$ is equal to 0; so from \cite{M-JSP11} the function $f|_\Lambda$ is expanding. In this expanding case we have that $B_n(z, \vp)$ is itself  a round ball of $\mu_0$-measure comparable to $\frac{1}{d^n}$ (from the estimates of equilibrium measures on Bowen balls in (\ref{bow})). At the same time the radius of this ball is comparable to $\vp|Df_u^n(z)|^{-1}$. Hence  the lower and upper pointwise dimensions of $\nu_i$ coincide and the pointwise dimension and Hausdorff dimension of $\nu_i$ are both equal to $$\delta_{\mu_0} = HD(\mu_0) = \frac{\log d}{\chi_u(\mu_0)}$$

\ c) In this case if $\Lambda$ is connected and $f$ is c-hyperbolic on $\Lambda$, it follows from \cite{M-Cam} that the preimage counting function is constant on $\Lambda$. Also $\Lambda$ cannot be written as a disjoint union of compact sets, so it is mixing for an iterate of $f$. So we can apply Theorem \ref{ptw} for the measure $\nu$ which, according to Theorem \ref{t1} is equal to $\mu_0$, and thus has entropy $\log d$. 

\end{proof}

\textbf{Remark.}
If $f$ is a smooth endomorphism hyperbolic on a basic set $\Lambda$, then by taking a smooth perturbation $g$ of $f$, it follows that $g$ has also a basic set $\Lambda_g$ on which it is hyperbolic (see \cite{Ru-carte89}). Also if $\Lambda$ is \textbf{connected} then $\hat \Lambda$ is connected, so from the conjugacy of $\hat f|_{\hat \Lambda}$  to $\hat g|_{\hat \Lambda_g}$,  $\hat \Lambda_g$ is connected and thus $\Lambda_g$ is connected too. However the dynamics of perturbations may be very \textbf{different}; for instance perturbations of toral endomorphisms are not conjugated necessarily to the original maps, and may even have infinitely many unstable manifolds through a given point. 

$\hfill\square$

For minimal c-hyperbolic sets of maps of degree 2 we can determine the possible values of the pointwise dimension of $\nu$;  recall that the preimage counting function is constant if $\Lambda$ is connected.

\begin{cor}\label{grad2}
Let $f$ be an Axiom A holomorphic map on $\mathbb P^2$ of degree $2$, which is c-hyperbolic on a connected minimal saddle set $\Lambda$. Then we have exactly one of the following two possibilities:

1) the preimage counting function of $f$ is equal to $1$ on $\Lambda$; then $f|_\Lambda$ is a homeomorphism and $$\delta_{\nu} = \log 2 \cdot (\frac{1}{\int \log|Df_u| d\nu} -  \frac{1}{\int\log |Df_s|d\nu})$$

2) or, the preimage counting function of $f$ is equal to 2 on $\Lambda$; then $f|_\Lambda$ is expanding and $$\delta_{\nu} = \log 2 \cdot \frac{1}{\int \log |Df_u| d\nu}$$
\end{cor}

\begin{proof}
If $f$ has Axiom A and $\Lambda$ is minimal then $\Lambda$ is terminal, thus from \cite{DJ} it follows that $h_{top}(f|_\Lambda) = \log 2$. 
Now if $\Lambda$ is connected and if there exists a neighbourhood $U$ of $\Lambda$ with $f^{-1}(\Lambda) \cap U = \Lambda$ then the preimage counting function is constant on $\Lambda$. If the preimage counting function is equal to $d'$ on $\Lambda$ it follows that  $d' \le 2$, otherwise from Misiurewicz-Przytycki Theorem (see \cite{KH}) we would have $h_{top}(f|_\Lambda) \ge \log d' > \log 2$ which is impossible, as we saw above.

So we either have $d' =1$ or $d'= 2$. In the first case we can apply Theorem \ref{ptw} for the potential $\phi \equiv 0$ since $\log d'< P(0) = \log d$ on $\Lambda$. In this case  $f|_\Lambda$ is a homeomorphism (like for instance the family of polynomial perturbations constructed in \cite{MU-HJM}). 

In the second case if $d'= 2$,  the stable dimension $\delta^s := HD(W^s_r(z) \cap \Lambda)$ is equal to the unique zero of the pressure function $t \to P(t\log |Df_s| - \log 2)$ (see \cite{M-Cam} and references therein); but since $h_{top}(f|_\Lambda) = \log 2$, the zero of this pressure function is indeed equal to 0. Then $\delta^s = 0$ and we can apply the result of \cite{M-JSP11} saying that in this case, $f$ must be expanding on $\Lambda$. 

From Theorem \ref{t1} the measure $\nu$  is equal to $\mu_0$, i.e the measure of maximal entropy. Therefore from Theorem \ref{ptw} the pointwise dimension of $\nu$ is $\delta_{\nu} = \log 2 \cdot \frac{1}{\int \log |Df_u| d\nu}$.

\end{proof}

\textbf{Remark and examples.}

 \ 1) \ First notice that the map $f_0(z, w) = (z^2+c, w^2)$ is 2-to-1 and expanding on $\Lambda_0= \{p_0(c)\} \times S^1$, where $p_0(c)$ is the fixed attracting point of $z \to z^2 +c$ for $|c|$ is small enough. \
We gave in \cite{MU-HJM} a class of examples of \textbf{perturbations} of $f_0$ by polynomial maps which are \textbf{homeomorphic on their respective basic sets}. This shows in particular that the preimage counting function is \textbf{not} necessarily preserved by perturbations. These are maps of type $$f_\vp(z, w) = (z^2 + c+ a\vp z + b\vp w + d\vp zw + e \vp w^2, w^2),$$ for $b \ne 0$, $|c|$ small and $0 < \vp < \vp(a, b, c, d, e)$. Then $f_\vp$ has a basic set $\Lambda_\vp$ (close to $\Lambda_0$), on which it is hyperbolic and has a homeomorphic restriction. \ 
For $0 < \vp < \vp(a, b, c, d, e)$ it also follows that $f_\vp$ is c-hyperbolic on $\Lambda_\vp$ since there are no critical points of $f_\vp$ in $\Lambda_\vp$ and since there exists a neighbourhood $U$ of $\Lambda_\vp$ such that $f_\vp^{-1}(\Lambda_\vp) \cap U = \Lambda_\vp$. Indeed for $\vp$ fixed if there were no such neighbourhood, then for any neighbourhood $V$ of $\Lambda_\vp$ there would exist a point $y \in V \setminus \Lambda_\vp$ with $f_\vp(y) \in \Lambda_\vp$. Thus for any $n >0$ there would exist points $y_n \in B(\Lambda_\vp, \frac 1n)$ with $f_\vp(y_n) = x_n \in \Lambda_\vp$,  and since $\Lambda$ is compact we can assume $x_n \to x \in \Lambda_\vp$. Since there are no critical points of $f_\vp$ in $\Lambda_\vp$, it follows that there exists a positive distance $\eta_0$ s. t $d(y_n, z_n) > \eta_0$, where $z_n$ is another preimage of $x_n$ belonging to $\Lambda_\vp$ (such $z_n$ must exist since $f_\vp(\Lambda_\vp) = \Lambda_\vp$). Then without loss of generality we can assume that $y_n \to y$ so $y \in \Lambda$ since $d(y_n, \Lambda) < \frac 1n$. But perhaps after passing to a subsequence,  $z_n \to z \in \Lambda$; then from above $d(y, z) > \eta_0/2$. But this is a contradiction since $f_\vp$ is 
homeomorphic on $\Lambda_\vp$. Hence there must exist a neighbourhood $U$ of $\Lambda_\vp$ satisfying \ $$f_\vp^{-1}(\Lambda_\vp) \cap U = \Lambda_\vp$$ 
Also notice that if we fix $a, b, c, d, e, \vp$ as above and perturb now $f_\vp$, we obtain another map $g$ which has a saddle basic set $\Lambda_g$ on which $g$ is hyperbolic and homeomorphic. 

\

\ 2) \ \textbf{Examples} of terminal sets can be obtained by perturbations of known examples $f$ and $\Lambda$; if $W^u(\hat \Lambda) \setminus \Lambda$ is contained in the union of basins of finitely many attracting cycles of $f$, then any small perturbation $g$ has a saddle basic set $\Lambda_g$ close to $\Lambda$, and $W^u(\hat \Lambda_g) \setminus \Lambda_g$ is also contained in the union of basins of attraction of $g$; hence $\Lambda_g$ is terminal too.   

Also the topological entropy of restrictions is preserved by perturbations, i.e $h_{top}(f|_\Lambda) = h_{top}(g|_{\Lambda_g})$. 
Thus by perturbing the known examples (like products of hyperbolic rational maps, or Ueda type examples of \cite{FS}, see \cite{FS}), we obtain more examples of terminal sets. As noticed before if $\Lambda$ is connected, then the basic set $\Lambda_g$ is connected too. And if $f$ is mixing on $\Lambda$ then $\hat f$ is mixing on $\hat \Lambda$; hence from the conjugacy on inverse limits, we obtain that $g$ is mixing on $\Lambda_g$ as well.
 
We can take for instance examples constructed by Ueda's method (see \cite{U}); if $\Phi: \mathbb P^1 \times \mathbb P^1 \to \mathbb P^2$ is the Segre map $\Phi([z_0: z_1], [w_0: w_1]) = [z_0w_0: z_1w_1: z_0w_1 + z_1w_0]$ and $f: \mathbb P^1 \to \mathbb P^1$ is a rational map then there exists $F:\mathbb P^2 \to \mathbb P^2$ holomorphic of the same degree as $f$, so that $\Phi(f, f) = F \circ \Phi$. If $f$ is hyperbolic on its Julia set $J(f)$ (i.e expanding), then $F$ is hyperbolic on basic sets of type \  $$\Lambda = \Phi(\{\text{periodic sink of} \ f\} \times J(f))$$ The saddle set   $\Lambda$ is terminal and topologically mixing for $F$. \
Let us consider now also a holomorphic \textbf{perturbation} $G$ of $F$ with a corresponding basic set $\Lambda_G$, which is close to $\Lambda$. From above it follows that $\Lambda_G$ is terminal and mixing saddle set for $G$. 
Consider also H\"older potentials $\phi$ on $\Lambda_G$  satisfying inequality (\ref{IN}) with respect to $G$; for instance, in the setting of Corollary \ref{hol} we can take $\phi$ sufficiently small in $\mathcal{C}^0$-norm s.t (\ref{IN}) is still satisfied.

Now for each such $\phi$ we have an equilibrium measure $\mu_\phi$ on $\Lambda_G$.
Then it follows that one can apply Theorem \ref{ptw},  Corollary \ref{ineq} and Corollary \ref{hol} in order to obtain the values of  $\mu_\phi$ on iterates of Bowen balls in $\Lambda_G$, and also in order to estimate the (upper/lower) pointwise dimensions of $\mu_\phi$. 

In particular we obtain information about the (upper/lower) pointwise dimensions for the measure $\mu_{0, G}$ of maximal entropy of the restriction $G|_{\Lambda_G}$. 

\

\textbf{Acknowledgements:}

John Erik Fornaess, \ \ \ \ \  johnefo\@@math.ntnu.no

Institutt for Matematiske Fag, 
  Sentralbygg 2, Alfred Getz Vei 1,

  7491 Trondheim,
  Norway

\

 Eugen Mihailescu, \ \ \ \ \  Eugen.Mihailescu\@@imar.ro

Institute of Mathematics of the Romanian Academy,
P. O. Box 1-764,
RO 014700,  

Bucharest, Romania.

\end{document}